\documentclass[11pt]{article}

\usepackage{amsmath}
\usepackage{amssymb}
\usepackage{amsthm}
\usepackage{caption}
\usepackage{setspace}
\usepackage{graphicx}
\usepackage{a4wide}
\usepackage{wrapfig}
\usepackage{tikz}
\usepackage{float}
\usepackage{listings}
\usepackage{sidecap}
\usepackage{bbm}
\usepackage{epstopdf}
\usepackage{ifthen}
\usepackage{tikz}
\usepackage{fancyhdr}
\usetikzlibrary{arrows}
\definecolor{cqcqcq}{rgb}{0.7529411764705882,0.7529411764705882,0.7529411764705882}
\definecolor{uuuuuu}{rgb}{0.26666666666666666,0.26666666666666666,0.26666666666666666}
\definecolor{qqccqq}{rgb}{0.0,0.8,0.0}
\definecolor{qqqqcc}{rgb}{0.0,0.0,0.8}
\definecolor{qqqqcc}{rgb}{0.,0.,0.8}
\definecolor{ffqqtt}{rgb}{1.,0.,0.2}

\makeatletter
\renewcommand\@biblabel[1]{#1.}
\makeatother

\newtheorem{thm}{Theorem}

\newtheorem{lem}[thm]{Lemma}

\def\eps{\varepsilon}
\providecommand{\p}{\partial}

\def\fref#1{(\ref{#1})}

\providecommand{\dr}{\Delta \rho}
\providecommand{\dev}{\Delta V}

\newenvironment{bottompar}{\par\vspace*{\fill}}{\clearpage}

\begin{document}

\title{Existence and uniqueness of solutions for a model of non-sarcomeric actomyosin bundles}

\author{Stefanie Hirsch\footnote{Fakult\"at f\"ur Mathematik, Universit\"at Wien, Austria} , 
Dietmar \"Olz\footnote{Courant Institute of Mathematical Sciences, New York University, U.S.A.} , 
Christian Schmeiser$^*$}
\rhead{SOLUTIONS FOR MODEL OF ACTOMYOSIN BUNDLES}
\date{\vspace{-5ex}}
\maketitle
\begin{abstract}
The model for disordered actomyosin bundles recently derived in \cite{Oelz} 
includes the effects of cross-linking of parallel and anti-parallel actin filaments,
their polymerization and depolymerization, and, most importantly, the interaction
with the motor protein myosin, which leads to sliding of anti-parallel filaments
relative to each other. The model relies on the assumption that actin filaments are
short compared to the length of the bundle. It is a two-phase model which treats
actin filaments of both orientations separately. It consists of quasi-stationary force
balances determining the local velocities of the filament families and of transport
equations for the filaments. Two types of initial-boundary value problems are 
considered, where either the bundle length or the total force on the bundle are
prescribed. In the latter case, the bundle length is determined as a free boundary.
Local in time existence and uniqueness results are proven. For the problem with given
bundle length, a global solution exists for short enough bundles. For small prescribed
force, a formal approximation can be computed explicitly, and the bundle length 
tends to a limiting value.

\begin{bottompar}
\end{bottompar}

\newpage

\end{abstract}

\section{Introduction}
In many biological processes such as wound healing, muscle contraction, or cytokinesis, contraction of \textit{actomyosin bundles} plays a central role. Actin is a polar protein that forms filaments by polymerization. If filaments of myosin II, a motor protein, bind to two 'anti-parallel' actin filaments with opposing plus ends (also called \textit{barbed ends} as opposed to the so-called \textit{pointed} or minus ends), it moves along both filaments towards their respective barbed ends, thus sliding the filaments past one another and changing the length of the actomyosin ensemble \cite{Vicente}.

In sarcomeres, the highly structured subunits of muscle cells, contractility has been the focus of research for decades and is well understood by now \cite{Huxley}. The dynamics of non-sarcomeric actomyosin structures (and what the minimal requirements for contraction are), however, is a field of ongoing research. Non-sarcomeric bundles include the contractile ring in cytokinesis \cite{Carvalho}, stress fibres in motile cells such as fibroblasts \cite{Small}, or rear bundles in fish epidermal keratocytes \cite{Svitkina}. They typically consist of actin filaments of varying and changing lengths and 
of both orientations, 
moving relative to each other, interspersed with myosin II filaments, cross-linkers such as $\alpha$-actinin (a protein that binds to two anti-parallel filaments \cite{Courson}), and fascin (which connects actin filaments of the same orientation \cite{Jayo}). Length changes are typically due to depolymerization.

A one-dimensional model for the dynamics of such disordered actomyosin bundles has been derived by the second author \cite{Oelz}, taking into account
the aforementioned protein dynamics (for details and an overview of other modelling approaches see \cite{Oelz} and the references therein). Its derivation starts on a microscopic level and describes each actin filament by the movement of its plus and minus ends.

A force-balance equation accounts for the influences of myosin, of cross-linking, and possibly of external forces (e.g. due to linkages to the substrate). Assuming that proteins act as transient elastic springs that bind and unbind in the overlapping region of filaments, an average over these
effects can be described by friction \cite{MilOel,OelSch}. 
An effective interaction coefficient between two filaments is derived, where the length of the overlap is modified by a factor reducing the 
interaction probability in thick bundles.

A continuous description is introduced by replacing individual filaments by a position (of the filament center of gravity) and length dependent density function. The force-balance equation now becomes an integral equation, where the effective interaction coefficient involves integrals over all the actin
material within the overlapping region.

A significant simplification of the model is achieved by the assumption that individual filaments are short compared to the length of the bundle. In the 
corresponding asymptotic limit the force balance equations are localized and turn into two coupled elliptic equations for the quasistationary
velocities of left- and right-moving filaments. The effective interaction coefficients now depend on moments of the actin densities with respect to 
filament length. On the other hand, the velocities appear in transport equations for the densities.

This derivation of the model does not consider ends of the bundle. Considering the model for a bundle of finite, time-dependent length, two-point
boundary conditions for the velocities and inflow boundary conditions for the densities are needed. Two choices seem natural: If the total force
acting at the ends of the bundle is prescribed, its length has to be treated as an unknown, resulting in a free boundary problem. If, on the other hand,
the dynamics of the bundle length is prescribed, the applied force can be computed as post-processing. These scenarios consider the force-velocity
relationship  of the contractile bundle either as a function of force or
of contraction rate, respectively.
The aim of this study is to examine the solvability of both problems. 

We shall first introduce the mathematical model in Section \ref{mathmodel}, then transform both problems to a fixed boundary by a coordinate transform in Section \ref{assumstrat}, where we also specify reqirements on initial and boundary data. Furthermore we explicitly solve the problem for the case, 
where no forces act on the tips. In Section \ref{fbvp}, the free boundary problem with prescribed force is considered. Since in general there is no control  
of the bundle length, solutions can only be expected to exist locally in time. Similarly, the solution behaviour necessary for the well posedness of the
initial-boundary value problem can only be expected for small enough forces. We therefore prove a local-in-time existence and uniqueness result
under the assumption of a small force. In Section \ref{fxbvp}, the problem corresponding to the situation with prescribed bundle length is treated. 
A local existence and uniqueness result can be shown, which is very similar to the free boundary problem. Under the additional assumption of 
small enough bundle length solutions are proven to exist globally. The local existence results are proven by decoupling the problems for the filament
densities and for the velocities and by an analysis of the resulting fixed point operator. The main difficulty is to avoid the occurrence of regions, where
all filaments of one of the two families are completely decomposed. This definitely happens after some time, if the bundle becomes too long, which
explains the shortness assumption for the global existence result. Finally, in Section 6 we return
to the problem with prescribed force and investigate a situation, where global existence can
be expected. For small force and bundle length and for time independent boundary data,
a formal approximation of the solution can be computed explicitly. Under the assumption of
a contractive bundle and a pulling force, convergence to a steady state with finite bundle
length is obtained.

\section{The mathematical model}\label{mathmodel}

Let the bundle at time $t\ge 0$ be located in the $x$-interval $[0,X(t)]$, and let $\rho^+(t,x,l)$ and, respectively, $\rho^-(t,x,l)$ be the densities of 
actin filaments (with their pointed ends directed in the positive and, respectively, the negative $x$ direction) with respect to their centers
$x\in[0,X(t)]$ and to their lengths $l\ge 0$. Note that this is a two-scale model, where the length variables $x$ and $l$ vary on different scales. 
By $V^\pm(t,x)$ we denote the velocities of the filaments centered at $x$ (independent of filament
length because of the strong friction between co-localized filaments of the same family), and by the given constant $s_l>0$ the difference between the depolymerization and the polymerization speeds, i.e. depolymerization dominates and filaments become shorter with time.

By actin-myosin interaction, the plus-filaments are expected to move to the right relative to the minus-filaments. Concerning the ends of the bundle, 
we want to describe a situation, where plus-filaments enter the bundle from the left end and leave it at the right end, and vice versa for the 
minus-filaments, meaning 
\begin{equation}\label{vel-ass}
  V^+(x=0)>0, \quad V^+(x=X) > \dot X, \quad V^-(x=0)<0, \quad V^-(x=X) < \dot X . 
\end{equation}
This will be a consequence of assumptions
on the data formulated below.  If it holds, the following initial-boundary value problem for the densities can be expected to be well posed for given 
$V^\pm$ and $X$:
\begin{eqnarray}\label{eq:prob}
\begin{array}{l}     
    \p_t\rho^\pm+\p_x(V^\pm \rho^\pm) - s_l\,\p_l\rho^\pm = 0 \,,\\
    \rho^+(t,0,l)=\rho_0^+(t,l),\qquad    \rho^-(t,X(t),l)=\rho_1^-(t,l) \,, \\
    \rho^\pm(0,x,l) = \rho_I^\pm(x,l) \,,
\end{array}
\end{eqnarray}
where the boundary data $\rho_0^+$, $\rho_1^-$ and initial data $\rho_I^\pm$ are given.
This subproblem is coupled to the problem for the velocities:
\begin{eqnarray}\label{eq:ellipticfull}
\begin{array}{ l}     
    \p_x\left(D^\pm\p_xV^\pm\right) \pm C \left(\eta - V^+ + V^-\right) = 0 \,,\\
    V^+(t,0)= u^+_0(t) \,,\qquad \p_xV^+(t,X(t))=0 \,,\\
    \p_xV^-(t,0)=0 \,,\qquad  V^-(t,X(t))= \dot X(t) - u^-_1(t) \,,
\end{array}
\end{eqnarray}
where the differential equation is a force balance, with the first term (reminiscent of the viscous term in fluid models) describing friction caused by filaments
of the same family as a consequence of the building and breaking dynamics of connections by bundling proteins. The second term models the interactions
between the two families with the contributions $C\eta$ from actin-myosin interaction and $C(V^+ - V^-)$ from the cross-linking of antiparallel
filaments. The parameter $\eta>0$ measures the strength of the actin-myosin interaction relative to the cross-linking, and
the viscosity coefficients and the interaction strength coefficient between antiparallel filaments are given by
\begin{equation}\label{eq:coeff}
  D^\pm = \frac{D_0\mu_1^\pm\mu_3^\pm}{\mu_1^+  + \mu_1^-} \,,\quad
  C = \frac{C_0\mu_1^+\mu_1^-}{\mu_1^+ + \mu_1^-} \,,\quad  \mu_j^\pm(t,x) = \int_0^\infty l^j \rho^\pm(t,x,l)dl\,,\,\, j=1,3\,,
\end{equation}
with $D_0, C_0>0$. 
Note that both coefficients increase only linearly (instead of quadratically) with the local total filament length, taking into account that in a thick bundle
interaction between two filaments becomes less likely.

Plus-Filaments are pushed into the bundle at $x=0$ with speed $u^+_0>0$, and minus-filaments at $x=X$ with (relative) speed $u^-_1>0$. 
The absence of forces on the plus-filaments at $x=X$ and on the minus-filaments at $x=0$ are described by the homogeneous Neumann boundary conditions. 

Adding the differential equations in \fref{eq:ellipticfull} shows that the quantity
$$
  F := D^- \p_x V^- + D^+ \p_x V^+
$$
is independent of $x$. It can be interpreted as the total force acting on the ends of the bundle, pulling it apart when positive.
In the following, two different situations will be considered. On the one hand, for given bundle length $X(t)$, \fref{eq:prob}, \fref{eq:ellipticfull}, \fref{eq:coeff} 
is a closed system for the computation of $\rho^\pm, V^\pm$, and the computation of the force can be considered as post-processing. On the other
hand, the force 
\begin{equation}\label{eq:force}
  F(t) = D^-(t,X(t)) \p_x V^-(t,X(t)) 
\end{equation}
may be considered as given, and the bundle length $X$ as an unknown. In this case \fref{eq:force} provides an additional equation for the determination of the free boundary $X$, and its initial position $X(0) = X_0> 0$ has to be prescribed. 

\section{Preliminary Assumptions and Strategy}\label{assumstrat}

In this section, our strategy for proving well posedness of the free boundary problem \fref{eq:prob}--\fref{eq:force} for the determination of
$(\rho^+,\rho^-,V^+,V^-,X)$ will be presented. It will be carried out in detail in the following section. The last section will be concerned with
the simpler proofs for the case with prescribed bundle length, where only the main differences will be highlighted.

As a preliminary step, the problem is transformed to a fixed domain by introducing the new variable
\begin{equation}\label{eq:scaling}
 y:= \frac{x}{X(t)} \in [0,1] \,.
 \end{equation}
The proof of a local existence result will be based on a fixed point iteration on the triple $(\rho^+(t,y,l),$ $\rho^-(t,y,l),X(t))$. Given these quantities, the
first step is the computation of the coefficients
\begin{equation}\label{eq:coeff-y}
  D^\pm = \frac{D_0\mu_1^\pm\mu_3^\pm}{\mu_1^+  + \mu_1^-} \,,\quad
  C = \frac{C_0\mu_1^+\mu_1^-}{\mu_1^+ + \mu_1^-} \,,\quad  \mu_j^\pm(t,y) = \int_0^\infty l^j \rho^\pm(t,y,l)dl\,,\,\, j=1,3\,.
\end{equation}
Then the velocities $V^+(t,y),V^-(t,y)$ are computed from the rescaled subproblem
\begin{eqnarray}\label{eq:ellfreeb}
\begin{array}{ l}     
    0=\p_y\left(D^\pm\p_yV^\pm\right)\pm X^2 C\left(\eta-V^++V^-\right),\\
    V^+(y=0)=u_0^+, \qquad   \p_y V^+(y=1)=0,\\
    \p_y V^-(y=0)=0,\qquad      (D^- \p_y V^-)(y=1) = X F.\\
\end{array}
\end{eqnarray}
A new version of $X(t)$ can be obtained by integration:
\begin{equation}\label{eq:movingb}
  \dot X= V^-(y=1)+u_1^-\,,\qquad X(0)=X_0\,.
\end{equation}
Finally, the new version of $(\rho^+(t,y,l),\rho^-(t,y,l))$ can be computed from
\begin{eqnarray}\label{eq:rhofreeb}
 \begin{array}{ l}
   \p_t\rho^\pm+\frac{1}{X}\left(V^\pm-(V^-(y=1)+u_1^-)y\right)\p_y\rho^\pm-s_l\p_l\rho^\pm=-\frac{1}{X}\rho^\pm \p_yV^\pm,\\
    \rho^+(y=0)=\rho_{0}^+, \qquad \rho^-(y=1)=\rho_1^-,\\
       \rho^\pm(0,y,l)=\rho_I^\pm(X_0 y,l).
 \end{array}
\end{eqnarray}
It is instructive to consider the case $F=0$, where \fref{eq:ellfreeb} has the simple explicit solution
\begin{equation}\label{explicit-V}
  \bar V^+(t,y) = u_0^+(t) \,,\qquad \bar V^-(t,y) = u_0^+(t) - \eta \,. 
\end{equation}
The bundle length can also be computed explicitly:
\begin{equation}\label{explicit-X}
  \bar X(t) = X_0 + \int_0^t (u_0^+(s) + u_1^-(s) - \eta)ds \,.
\end{equation}
In view of \fref{vel-ass}, we now pose our assumptions on the boundary data for the velocities: We shall assume the existence of a positive constant 
$\delta\le \eta/2$, such that
\begin{equation}\label{BC-ass}
  0 < \delta \le u_0^+,\,u_1^- \le \eta-\delta \,,\qquad u_0^+,u_1^- \in C([0,\infty))\,.
\end{equation}
The problems for the densitites can then be solved by the method of characteristics. We introduce the age $\tau^+(y,t)\in [0,t]$ of a plus-filament
with position $y$ at time $t$, which has entered the bundle at $y=0$. Because of the positivity of $u_0^+$, it is determined uniquely by
$$
   \bar X(t)y = \int_{t-\tau^+(y,t)}^t u_0^+(s)ds \,.
$$
The density of plus-filaments is then given by
\begin{equation}\label{explicit-rho+}
  \bar\rho^+(t,y,l) = \left\{ \begin{array}{ll} \rho_I^+ \left( \frac{\bar X(t)}{X_0}y - \frac{1}{X_0} \int_0^t u_0^+(s)ds, l+s_l t \right) \,, & 
              \mbox{for }   \bar X(t)y > \int_0^t u_0^+(s)ds \,,\\
     \rho_0^+(t-\tau^+(y,t), l+s_l\tau^+(y,t)) \,,& \mbox{for }   \bar X(t)y < \int_0^t u_0^+(s)ds \,. \end{array} \right.
\end{equation}
Analogously, we obtain for the minus-filaments
\begin{equation}\label{explicit-rho-}
  \bar\rho^-(t,y,l) = \left\{ \begin{array}{ll} \rho_I^- \left( \frac{\bar X(t)}{X_0}(y-1) + 1 + \frac{1}{X_0} \int_0^t u_1^-(s)ds, l+s_l t \right) \,, & 
              \mbox{for }   \bar X(t)(1-y) > \int_0^t u_1^-(s)ds \,,\\
     \rho_1^-(t-\tau^-(y,t), l+s_l\tau^-(y,t)) \,,& \mbox{for }   \bar X(t)(1-y) < \int_0^t u_1^-(s)ds \,, \end{array} \right.
\end{equation}
where $\tau^-(y,t)\in [0,t]$ (the age of a minus-filament with position $y$ at time $t$, which has entered the bundle at $y=1$) is determined by
$$
   \bar X(t)(1-y) = \int_{t-\tau^-(y,t)}^t u_1^-(s)ds \,.
$$
The explicit solution \fref{explicit-V}, \fref{explicit-X}, \fref{explicit-rho+}, \fref{explicit-rho-} already reveals several obstacles for the 
existence of global solutions. Obviously, the length of the bundle \fref{explicit-X} might shrink to zero in finite time, but it might also grow above all bounds,
which then also holds for the ages $\tau^\pm$.
If, in this case, the boundary data for the densities have compact supports in terms of the filament length $l$, the filaments pushed in 
at one end might be completely depolymerized before reaching the other end. In other words, for certain $(t,y)$, $\rho^\pm(t,y,\cdot)=0$ 
will hold, whence the diffusivity $D^\pm(t,y)$ vanishes and the elliptic nature of the equation for $V^\pm$ is lost. In the following, this situation will be avoided by smallness assumptions on the data (although the bundle should be able to
support pulling forces as long as the supports of $D^+(t,\cdot)$ and $D^-(t,\cdot)$ overlap and their union covers $[0,1]$). 
In the following, we collect our boundedness, regularity, and compatibility assumptions on the data for the densities. We shall use the notation
$\Omega_T := (0,T)\times (0,1)$ and assume the existence of positive constants $\alpha_0\le \beta_0$, $\underline{L} < \overline{L}$, and $M$,
such that 
\begin{eqnarray}\label{rho-ass}  
\begin{array}{ll}
     \rho_0^+(t,l),\,\rho_1^-(t,l),\,\rho_I^\pm(y,l) \ge \alpha_0\,,\qquad &\mbox{for } (t,y,l)\in\Omega_\infty\times[0,\underline{L}] \,,\\
     \rho_0^+,\,\rho_1^-,\,\rho_I^\pm \le \beta_0\,,\qquad &\mbox{in } \Omega_\infty\times [0,\infty) \,,\\
     \rho_0^+ = \rho_1^- = \rho_I^\pm = 0\,,\qquad &\mbox{in } \Omega_\infty\times [\overline{L},\infty) \,, \\
     |\p_y \rho^\pm_I|,\, |\p_t \rho_0^+|,\, |\p_l \rho_0^+|,\, |\p_t \rho_1^-|,\, |\p_l \rho_1^-| \le M \,, \qquad &
            \mbox{in } \Omega_\infty\times [0,\infty) \,, \\ 
     \rho_0^+(0,l) = \rho_I^+(0,l) \,,\quad \rho_1^-(0,l) = \rho_I^-(1,l) \,,\qquad &\mbox{for } l\ge 0\,.
\end{array}  
\end{eqnarray}

\begin{center}
\begin{figure}
\includegraphics[scale=.2]{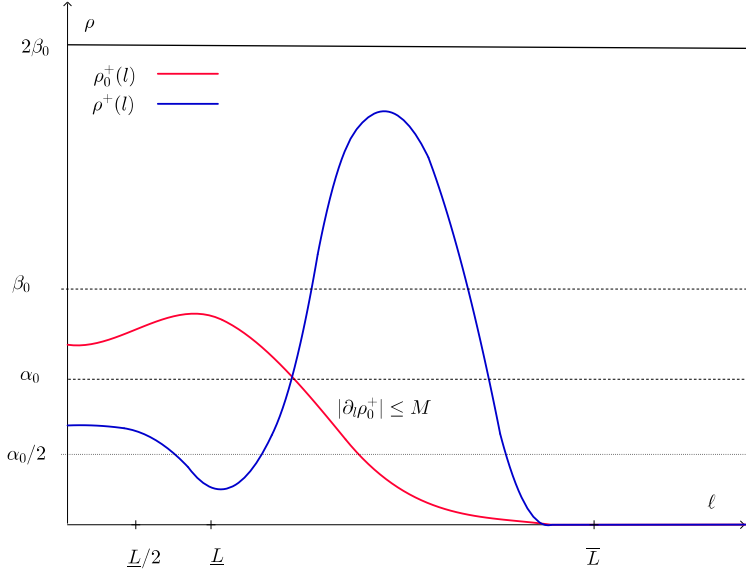}
\caption{The properties of $\rho^+\in R_{\rho,T}$ and the boundary datum $\rho_0^+$ as functions of $l$}
\end{figure}
\end{center}

The first of these assumptions guarantees that the complete depolymerization described above is avoided at least in the explicit solutions \fref{explicit-rho+},
\fref{explicit-rho-} for short times $t<\underline{L}/s_l$.

Another difficulty results from a possible violation of \fref{vel-ass} for large applied forces. The existence result in the following section will therefore 
be local in time under the additional assumption of small enough $F(t)$.

\section{Prescribed Force -- Local Existence and Uniqueness}\label{fbvp}
In this section, we consider the situation where the force $F(t)$ acting on the tips of the bundle is prescribed, and the problem 
\fref{eq:coeff-y}--\fref{eq:movingb} is expected to determine the coefficients $D^\pm,C$, the velocities $V^\pm$, the densities $\rho^\pm$, and the
bundle length $X$. We shall prove local existence and uniqueness of solutions for small enough $F$.

Collecting our assumptions on the data, we assume $C_0,D_0,\eta,s_l,X_0>0$, \fref{BC-ass}, and \fref{rho-ass}. When we say in the following that a
quantity depends on the data, it means that it depends on all the constants appearing in these assumptions. For a time $T>0$ to be chosen below,
the fixed point operator discussed in the previous section will act on the set $R_{\rho,T}^2 \times R_{X,T}$ with
\begin{align*}
  R_{\rho,T}:= \{ \rho\in L^\infty(\Omega_T\times(0,\infty)): \,& \rho(t,y,l)\ge\alpha_0/2 \mbox{ for } 0\le l\le \underline L/2,\, \rho^\pm(t,y,l)\le 2\beta_0 
     \mbox{ for } l\ge 0,\\
     & \rho^\pm(t,y,l)=0 \mbox{ for } l\ge \overline L,\, |\p_y\rho^\pm(t,y,l)| \le \gamma \mbox{ for }l\ge0\} \,,
\end{align*}
\[
  R_{X,T}:=\{X\in C([0,T]): \, X_0/2\leq X(t)\leq 2 X_0\} \,,
\]
where $\gamma>0$ depends on the data in a way defined below.

\begin{lem}\label{lem:CDfreeb}
Let, for an arbitrary $T>0$, $(\rho^+,\rho^-) \in R_{\rho,T}^2$. Then there exist $\underline\kappa, \overline\kappa, c>0$, depending on the data, 
such that $C$ and $D^\pm$, defined by \fref{eq:coeff-y}, satisfy 
\begin{equation} \label{CD-prop}
  \underline\kappa \le C,\,D^\pm \le \overline\kappa \,,\qquad |\p_y D^\pm|\le c \,,\qquad \mbox{in } \Omega_T\,.
\end{equation}
\end{lem}

\begin{proof}
The proof is straightforward using
$$
  \frac{(\underline{L}/2)^{j+1}}{j+1} \frac{\alpha_0}{2} \le \mu_j^\pm \le \frac{\overline{L}^{j+1}}{j+1} 2\beta_0 \,.
$$
\end{proof}

\begin{lem}\label{lem:Vfreeb}
Let, for an arbitrary $T>0$, $X\in R_{X,T}$ and let $C,D^\pm$ satisfy \fref{CD-prop}. 
Then \fref{eq:ellfreeb} has a unique solution $(V^+,V^-)\in L^\infty(0,T; W^{2,\infty}(0,1))^2$ satisfying 
\begin{equation}\label{eq:V-boundsfreeb}
  | V^+ - u_0^+|\leq \Gamma |F| \,,\qquad | V^- + \eta - u_0^+|\leq \Gamma |F| \,,\qquad |\p_yV^\pm|, |\p_y^2V^\pm|\leq c+\Gamma|F| \,,\qquad 
  \mbox{in } \Omega_T\,,
\end{equation}
with $\Gamma, c>0$ depending on the data.
\end{lem}
\begin{proof}
We introduce the new unknowns
\[
  U^+(t,y):=V^+(t,y)-u_0^+(t) \,,\qquad U^-(t,y):=V^-(t,y)-u_0^+(t)+\eta+F(t)X(t)\int_y^1\frac{\tilde y}{D^-(t,\tilde y)}d\tilde y \,,
\]
satisfying the differential equations
\begin{eqnarray*}
  &&\p_y(D^+\p_yU^+)- X^2C(U^+-U^-)=F X V_r \,,\\
&&\p_y(D^-\p_yU^-)+ X^2C(U^+-U^-)=-FX\left( V_r+1\right) \,,
\end{eqnarray*}
with
\[
  V_r(t,y):=C(t,y)X(t)^2\int_y^1\frac{\tilde y}{D^-(t,\tilde y)}d\tilde y \,,
\]
and the homogeneous boundary conditions
\[
  U^+(y=0)=0 \,,\quad \p_y U^+(y=1)=0 \,,\quad \p_y U^-(y=0)=0\,,\quad \p_y U^-(y=1)=0 \,.
\]
Accordingly, for each fixed $t\in [0,T]$, we shall construct solutions $U = (U^+,U^-)$ in the Hilbert space
\[
  H_\p:=\left\{\phi=(\phi^+,\phi^-)\in H^1(0,1)^2 :\,\phi^+(0)=0\right\} \,,
\]
equipped with the natural norm
\[
  ||\phi||^2_{H_\p}:=||\phi^+||^2_{H^1(0,1)}+||\phi^-||^2_{H^1(0,1)}.
\]
The problem for $U\in H_\p$ can then be written in the weak form $a(U,\phi)=G(\phi)$ for all $\phi \in H_\p$ with
\begin{eqnarray*}
 a(U,\phi) &:=& \int_0^1 \left( D^+\p_y U^+\p_y\phi^+ + D^-\p_y U^-\p_y\phi^- + X^2C(U^+ - U^-)(\phi^+ - \phi^-) \right)dy \,,\\
G(\phi) &:=& -FX\int_0^1\left(V_r \phi^+-(V_r + 1)\phi^-\right)dy \,.
\end{eqnarray*}
As a consequence of the assumptions on $X$, $C$, and $D^-$, $V_r$ is bounded in terms of the data, and there exists a constant $c$ depending 
on the data, such that the linear functional $G$ is bounded by
\[
  \|G\|_{H_\p^\prime} \le c|F| \,.
\]
In order to show that the bilinear form $a$ is coercive, we use the Poincar\'e inequality
\[
  \|U^+\|_{L^2(0,1)} \le \|\p_y U^+\|_{L^2(0,1)} \qquad\mbox{for } U \in H_\p \,,
\]
and the elementary inequality 
\[
  A u^2+B(u-v)^2\geq\frac{AB}{A+2B}(u^2+v^2) \,,\qquad\mbox{for } A,B,u,v>0 \,.
\]
We therefore have
\begin{eqnarray*}
  a(U,U) &\ge& \underline\kappa \int_0^1\left( (\p_y U^+)^2 + (\p_y U^-)^2 + \frac{X_0^2}{4}(U^+ - U^-)^2 \right)dy\\
  &\ge& \underline\kappa \int_0^1\left( \frac{1}{2} (\p_y U^+)^2 + (\p_y U^-)^2 +\frac{1}{2}(U^+)^2+\frac{X_0^2}{4}(U^+ - U^-)^2\right)dy\\
  &\ge& \underline\kappa \min\left\{ \frac{1}{2}, \frac{X_0^2}{4(1+X_0^2)}\right\}\|U\|_{H_\p}^2 \,.
\end{eqnarray*}
Continuity of $a$ with a bound depending on the data is straightforward. The Lax-Milgram Lemma implies existence and uniqueness of a weak
solution $(V^+,V^-)$ of \fref{eq:ellfreeb} and the bound
\[
  \|V^+-u_0^+\|_{H^1(0,1)}+\|V^--u_0^++\eta\|_{H^1(0,1)}\le c |F| \,,
\]
with $c$ depending on the data.
The Sobolev embedding $H^1(0,1) \hookrightarrow L^\infty(0,1)$ completes the proof of the first two estimates in \fref{eq:V-boundsfreeb}.

The boundedness of $\p_y V^\pm$ is a consequence of the formula
\[
  \p_yV^-(t,y)=\frac{X(t)^2}{D^-(t,y)}\int_0^yC(t,\tilde y)\left(\eta-V^+(t,\tilde y)+V^-(t,\tilde y)\right)d\tilde y \,,
\]
with a corresponding version for $\p_y V^+$. The boundedness of the second order derivatives for bounded $\p_y D^\pm$ follows from elliptic regularity.
\end{proof}

Finally, we consider the problems \fref{eq:movingb} and \fref{eq:rhofreeb} for the bundle length and, respectively, the densities. This is the point where smallness of the applied force will be needed.
We require that, with the velocities obtained in Lemma \ref{lem:Vfreeb}, the $y$-characteristics of the equation for $\rho^+$ (the plus-characteristics) 
enter the domain at $y=0$ and leave it at $y=1$, and vice versa for $\rho^-$. The velocity of the plus-characteristics at $y=0$ is 
$V^+(y=0) = u_0^+ \ge \delta > 0$ by \fref{BC-ass}. At $y=1$, we have by \fref{eq:V-boundsfreeb} and again \fref{BC-ass},
\[
  V^+(y=1) - V^-(y=1) - u_1^- \ge \eta - u_1^- - 2\Gamma|F| \ge \delta - 2\Gamma|F| \,.
\]
Similarly, for the velocity of the minus-characteristics at $y=1$, $-u_1^- \le -\delta < 0$, and at $y=0$
\[
  V^-(y=0) \le u_0^+ - \eta + \Gamma|F| \le -\delta + \Gamma|F| \,.
\]
Obviously, the assumption
\begin{equation} \label{F-ass}
   |F(t)| \le \frac{\delta}{4\Gamma} \qquad\mbox{for } t\in [0,T] \,,
\end{equation}
guarantees the desired properties. Apart from it, the time interval will have to be short enough for the fixed point operator to be a self map.

\begin{lem}\label{lem:X}
There exists $T>0$ depending on the data such that, if $V^\pm$ and $F$ satisfy \fref{eq:V-boundsfreeb} and, respectively, \fref{F-ass}, the
unique solution $X$ of \fref{eq:movingb} satisfies $X\in R_{X,T}$.
\end{lem}
\begin{proof}
The result follows immediately, since the assumptions imply a bound of $\dot X$ depending on the data (actually $|\dot X| \le \eta$).
\end{proof}

\begin{lem}\label{lem:rhofreeb}
There exists $T>0$ depending on the data such that, if $X\in R_{X,T}$, and $V^\pm$ and $F$ satisfy \fref{eq:V-boundsfreeb} and, respectively, 
\fref{F-ass}, the problem \fref{eq:rhofreeb} has a unique solution $\rho^\pm \in R_{\rho,T}$ with $\gamma>0$ depending on the data.
\end{lem}

\begin{proof}
Since, as a consequence of \fref{F-ass}, the plus- and minus-characteristics have the desired sign properties as discussed above, and by the regularity 
of $V^\pm$, the problem can be solved by the method of characteristics.
The condition $s_l>0$ implies that the $l$-characteristics are outgoing at the boundary $l=0$, such that no boundary conditions are required there
and the support of $\rho^\pm$ in terms of $l$ shrinks. 
The characteristic equation
$$
  \frac{d}{dt}\rho^\pm = -\frac{1}{X}\rho^\pm \p_yV^\pm \,,
$$ 
and the bound $|\p_y V^\pm/X| \le c$ depending on the data already imply the required inequalities for $T\le \underline L/(2s_l)$ and 
$e^{cT}\le 2$, except the bound on $\p_y\rho^\pm$. The latter is shown
by differentiating the differential equation with respect to $y$ and solving for $\p_y\rho^\pm$ again by the method of characteristics.
A boundary condition for $\p_y \rho^+$ is obtained from the differential equation:
$$
  \p_y \rho^+(y=0) = -\frac{X(\p_t \rho_0^+ -s_l\p_l \rho_0^+) + \rho_0^+ \p_y V^+(y=0)}{V^+(y=0)} \,.
$$
The right hand side and the initial data $\p_y\rho^+(t=0)=\p_y\rho^+_I$ are bounded by \fref{rho-ass} and by the assumption on $V^+$.
An exponentially increasing bound for $\p_y\rho^+$ (as for $\rho^\pm$) follows. The bound for $\p_y \rho^-$ is shown analogously.
\end{proof}

We now prove Lipschitz continuity of the maps defined by the previous four lemmas. For this purpose, we need norms for the solution components.
We choose
\begin{eqnarray*}
  &&\|\rho^\pm\|_{\rm density} := \sup_{t\in(0,T)} \|\rho^\pm(t,\cdot,\cdot)\|_{L^2((0,1)\times(0,\infty))} \,,\qquad 
     \|X\|_{\rm length} := \sup_{t\in(0,T)} |X(t)| \,,\\
  &&\|V^\pm\|_{\rm vel} := \sup_{t\in(0,T)} \|V^\pm(t,\cdot)\|_{H^1(0,1)} \,,\qquad \|D^\pm\|_{\rm coeff} := \sup_{t\in (0,T)} \|D^\pm(t,\cdot)\|_{L^2(0,1)} \,.
\end{eqnarray*}
Differences will be denoted by $\Delta\rho^\pm := \rho^\pm_1 - \rho^\pm_2$ and analogously for the other components $X,V^\pm,C,D^\pm$.

\begin{lem}\label{lem:Lipfreeb}
a) \label{Deltacoeff} Let $\rho^\pm_1$ and $\rho^\pm_2$ satisfy the assumptions of Lemma \ref{lem:CDfreeb}. Then there exists a constant $c>0$ depending on the data, such that the corresponding coefficients $(D^\pm_1,C_1)$ and $(D^\pm_2,C_2)$ defined by \fref{eq:coeff-y} satisfy
\begin{eqnarray*}
  && \|\Delta D^+\|_{\rm coeff} + \|\Delta D^-\|_{\rm coeff} + \|\Delta C\|_{\rm coeff}  \le c\left( \|\Delta \rho^+\|_{\rm density} + 
     \|\Delta \rho^-\|_{\rm density} \right) \,.
\end{eqnarray*}
b) Let $(D^\pm_1,C_1,X_1)$ and $(D^\pm_2,C_2,X_2)$ satisfy the assumptions of Lemma \ref{lem:Vfreeb} and let \fref{F-ass} hold. Then there exists 
a constant $c>0$ depending on the data, such that the corresponding solutions $V^\pm_1$ and $V^\pm_2$ of \fref{eq:ellfreeb} satisfy
\begin{eqnarray*}
  && \|\Delta V^+\|_{\rm vel}+\|\Delta V^-\|_{\rm vel}  \le c\left( \|\Delta X\|_{\rm length} + \|\Delta C\|_{\rm coeff} + 
     \|\Delta D^+\|_{\rm coeff} + \|\Delta D^-\|_{\rm coeff}\right) \,.
\end{eqnarray*}
c) Let $V_1^-$ and $V_2^-$ satisfy the assumptions of Lemma \ref{lem:X} and let \fref{F-ass} hold. Then there exists 
a constant $c>0$ depending on the data, such that the corresponding solutions $X_1$ and $X_2$ of \fref{eq:movingb} satisfy
\[
   \|\Delta X\|_{\rm length} \le  cT \|\Delta V^-\|_{\rm vel} \,.
\]
d) Let $(X_1,V^\pm_1)$ and $(X_2,V^\pm_2)$ satisfy the assumptions of Lemma \ref{lem:rhofreeb} and let \fref{F-ass} hold. Then there exist constants $c_1,c_2>0$ depending on the data, such that the corresponding solutions $\rho^\pm_1$ and $\rho^\pm_2$ of \fref{eq:rhofreeb} satisfy
\begin{eqnarray*}
  \|\Delta\rho^+\|_{\rm density} + \|\Delta\rho^-\|_{\rm density} \le c_1\left(e^{c_2T}-1\right) 
  \left( \|\Delta V^+\|_{\rm vel} + \|\Delta V^-\|_{\rm vel} + \|\Delta X\|_{\rm length} \right) \,.
\end{eqnarray*}
\end{lem}

\begin{proof}
a) Noting that, by the boundedness of the support, the $l$-moments of the densitites can be bounded in terms of the $L^2_l$-norm, the proof is 
a straightforward computation.\\
b) We proceed analogous to the existence proof, and introduce auxiliary functions
\[U_i^+(t,y):=V_i^+(t,y)-u_0^+(t),\,\, U_i^-(t,y):=V_i^-(t,y)-u_0^+(t)+\eta+F(t)\underbrace{X_i(t)\int_y^1\frac{\tilde yd\tilde y}{D_i^-(t,\tilde y)}}_{=:V_r^i(y)}.\]
Then $\Delta U:=(\Delta U^+,\Delta U^-)$ satisfies $a\left(\Delta U,\phi\right)=\overline G(\phi)$ for all $\phi:=(\phi^+,\phi^-)\in H_\p$, where the space $H_\p$, together with the corresponding norm, and the bilinear form $a$ are defined in the same way as in Lemma \ref{lem:Vfreeb}. The functional $\overline G$ is given by
\begin{align*}
 \overline G(\phi)=\int_0^1\Biggl(&-\Delta C((U_2^+-U_2^-)+F V_r^1)(\phi^+-\phi^-)-\frac{\Delta D^+}{X_1^2}\p_yU_2^+\p_y\phi^+
       -\frac{\Delta D^-}{X_1^2}\p_yU_2^-\p_y\phi^-\\
 &+\Delta X\frac{X_1+X_2}{X_1X_2}\left(D_2^+\p_yU^+_2\p_y\phi^++D_2^-\p_yU^-_2\p_y\phi^-\right)-\Delta X\frac{F}{X_1X_2}\phi^-\\
 &-(\phi^+-\phi^-)F C_2\int_y^1\tilde y\left(\frac{\Delta X}{D_2^-(\tilde y)}-\Delta D^-(\tilde y)\frac{X_1}{D_1^-(\tilde y)D_2^-(\tilde y)} \right)
    d\tilde y\Biggr)dy \,.
\end{align*}
We derive an estimate on this functional by using the boundedness of the functions $U_2^\pm,\p_yU_2^\pm$, $X_{1,2}$, $D_2^\pm$, and $C_2$. The Hilbert space norms, together with the $L_y^2-$norms of the coefficients $\Delta D^\pm,\Delta C$, are extracted with the Cauchy-Schwarz inequality as well as the elementary inequalities $(a\pm b)^2\leq2(a^2+b^2)$ and $\sqrt{x+y}\leq\sqrt x+\sqrt y\leq2\sqrt{x+y}$  for $x,y\ge 0$. \\
c) Since $\Delta \dot X=\dev^-(t,1)$, the proof is straightforward using the Sobolev embedding $H^1(0,1) \hookrightarrow L^\infty(0,1)$, and 
integrating the differential equation with $\Delta X(t=0)=0$.\\
d) We drop the superscript $\pm$, since the proof is the same for both cases. Multiplication of the difference of the equations for $\rho_1$ and $\rho_2$ by $2\Delta\rho$ and integration with
respect to $l$ and $y$ gives
\begin{eqnarray}
  \frac{d}{dt} \int_0^1 \int_0^\infty (\Delta\rho)^2 dl\,dy &=&2\int_0^1\int_0^\infty A\dr \,dl\,dy    - \frac{1}{X_1}\int_0^1 \int_0^\infty 
   \left(\p_yV_1+\dot X_1\right)(\dr)^2 dl\,dy \nonumber\\
  &&  -  \frac{1}{X_1} \int_0^\infty \left(V_1-\dot X_1y\right)(\dr)^2\Bigm|_{y=0}^1 dl \, -s_l \int_0^1 \Delta\rho(l=0)^2 dy \,,\label{eq:deltarho}
\end{eqnarray}
where in the first term we have used the abbreviation
\[
   A:=-\frac{\p_y\rho_2}{X_1}\dev-\frac{\rho_2}{X_1}\p_y\dev+\frac{y\p_y\rho_2}{X_1}\Delta \dot X+\Delta X\frac{1}{X_1X_2}
      \left((V_2-X_2y)\p_y\rho_2+\rho_2\p_yV_2\right) \,.
\]
The terms in the second line are nonpositive, which is obvious for the second. Concerning the first term, in the plus-equation $V_1-\dot X_1y$ is positive and $\Delta\rho(y=0)=0$, and for the minus-equation $V_1-\dot X_1y$ is negative and $\Delta\rho(y=1)=0$. Since $X_{1,2}$, $V_{1,2}$, $\p_y V_{1,2}$, $\rho_{2}$, and $\p_y\rho_{2}$ are bounded, and because $\Delta \dot X$ can be bounded by $\Delta V^-$ due to c), this implies
\begin{eqnarray*}
  \frac{d}{dt} \|\Delta\rho\|_{L^2((0,1)\times(0,\infty))} \le 
      c_1\|\Delta\rho\|_{L^2((0,1)\times(0,\infty))} 
  + c_2\left(\|\Delta V^+\|_{H^1(0,1)}+\|\Delta V^-\|_{H^1(0,1)}+\|\Delta X\|_{L^\infty_t}\right) \,.
\end{eqnarray*}
In the estimate of the first term the fact that $\Delta\rho$ has bounded support as a function of $l$ has been used. The proof of d) 
is completed by an application of the Gronwall inequality and noting that $\Delta\rho(t=0)=0$ holds.
\end{proof}

\begin{thm}
Let $C_0,D_0,\eta,s_l,X_0>0$, \fref{BC-ass}, \fref{rho-ass}, and \fref{F-ass} hold. Then there exists $T>0$, such that problem 
\fref{eq:prob}--\fref{eq:force} has a unique solution 
$$
  (\rho^+,\rho^-,V^+,V^-,X) \in L^\infty\left(0,T;\, W^{1,\infty}_x(L^\infty_l)^2\times (W^{2,\infty}_x)^2 \times \mathbb{R}\right) \,,
$$
such that $\rho^+,\rho^-> 0$ are uniformly bounded away from zero in $\{(t,x,l):\, 0<t<T,\,0<x<X(t),\,0<l<\underline L/2\}$,
$V^+ > \max\{0,\dot X\}$, and $V^- < \min\{0,\dot X\}$. 
\end{thm}

\begin{proof}
By first scaling the system according to \fref{eq:scaling}, Lemma \ref{lem:Lipfreeb} a)--d) shows that the fixed point map on $(R_{\rho,T})^2\times R_{X,T}$ 
is Lipschitz continuous with respect to the 
$L^\infty_t((L^2_{y,l})^2\times\mathbb{R})$-topology. By Lemma \ref{lem:Lipfreeb} c), d), it is a contraction for $T$ small enough, proving existence and uniqueness of the scaled version and therefore also of the original problem. 
The properties of the solution are a consequence of the Lemmas \ref{lem:CDfreeb}, \ref{lem:X}, and \ref{lem:rhofreeb}.
\end{proof}

\section{Prescribed bundle length -- existence and uniqueness}\label{fxbvp}

In this section, $X(t)$ will be assumed as given, satisfying
\begin{equation}\label{X-ass}
  X \in C^1([0,\infty)) \,,\qquad 0< X_0/2\leq X(t)\le 2 X_0  \quad\mbox{for } t\ge 0\,.
\end{equation}
This assumption removes one of the obstacles for global existence in the preceding section with the consequence that we prove a 
global existence result at the end of this section. The problem \fref{eq:ellfreeb} for the velocities is replaced by
\begin{eqnarray}\label{eq:ell}
\begin{array}{ l}     
    0=\p_y\left(D^\pm\p_yV^\pm\right)\pm X^2 C\left(\eta-V^++V^-\right) \,,\\
    V^+(y=0)=u_0^+, \qquad   \p_y V^+(y=1)=0 \,,\\
    \p_y V^-(y=0)=0,\qquad  V^-(y=1) = -u_1^- + \dot X \,,\\
\end{array}
\end{eqnarray}
and the system \fref{eq:coeff-y}, \fref{eq:ell}, \fref{eq:rhofreeb} has to be solved for $C$, $D^\pm$, $V^\pm$, and $\rho^\pm$.
The last equation in  \fref{eq:ellfreeb} can be used as a posteriori information, to calculate the force the bundle exerts on its neighborhood.

Since the subproblem \fref{eq:rhofreeb} for the densities $\rho^\pm$ did not change, we will assume the same bounds \fref{rho-ass}  on the data 
as before. For \fref{eq:rhofreeb} to be well posed, the assumption \fref{BC-ass} on the boundary data for the velocities will be replaced by
\begin{equation}\label{V-ass1}
  \delta + \dot X(t)_+ \le u_0^+(t),\,u_1^-(t) \le \eta - \delta - (-\dot X(t))_+ \,,\qquad t\ge 0\,,
\end{equation}
with $\delta>0$. Note that this implicitly contains the assumption $|\dot X| \le \eta-2\delta$ on the time changes of the bundle length.

The local existence proof follows the strategy of the preceding section. Therefore we shall only concentrate on the main differences. 
The fixed point iteration now acts on $(\rho^+,\rho^-)\in (R_{\rho,T})^2$. The first step is again the computation of the friction
coefficients $C,D^\pm$, whose properties are given in Lemma \ref{lem:CDfreeb}.
In the elliptic problem \fref{eq:ell} for the velocities, an inhomogeneous Neumann boundary condition has been replaced by a Dirichlet condition. 
In the proof of the following result, this permits direct estimates by the maximum principle, replacing the $L^\infty$-estimates by Sobolev embedding 
of the preceding section.

\begin{lem}\label{lem:Vfixedb}
Let, for an arbitrary $T>0$, $C,D^\pm$ satisfy \fref{CD-prop}, and let \fref{V-ass1} hold. 
Then \fref{eq:ell} has a unique solution $(V^+,V^-)\in L^\infty(0,T; W^{2,\infty}(0,1))^2$ satisfying 
\begin{eqnarray}
  && \delta + \dot X(t)_+ \le V^+(t,y) \le \eta - \delta - (-\dot X(t))_+ \,,\quad
  -\eta + \delta + \dot X(t)_+ \le V^-(t,y) \le - \delta - (-\dot X(t))_+ \,,\nonumber\\
  && |\p_yV^\pm|, |\p_y^2V^\pm|\leq c\,,\qquad\mbox{in } \Omega_T \,,\label{V-est}
\end{eqnarray}
with $c$ depending only on the data (which means the same as in the preceding section).
\end{lem}

\begin{proof}
As in Lemma \ref{lem:Vfreeb}, existence and uniqueness are a consequence of the Lax-Milgram lemma, now for the new unknowns
$$
  (W^+,W^-) := (V^+ - u^+_0, V^- - \dot X + u^-_1) \in \hat H_\partial\,,
$$
with 
$$
  \hat H_\partial := \{ \phi=(\phi^+,\phi^-)\in H^1(0,1)^2:\, \phi^+(0) = \phi^-(1) = 0\} \,.
$$
The necessary estimates are very similar and, thus, omitted.
On the other hand, with $V_0=u^+_0$, $V_1=\dot X - u^-_1$, the maximum principle implies
\begin{eqnarray*}
  && \min\left\{V_0(t), \eta + \min_{[0,1]} V^-(t,\cdot) \right\} \le V^+(t,y) \le \max\left\{ V_0(t), \eta + \max_{[0,1]} V^-(t,\cdot) \right\} \,,\\
  && \min\left\{ V_1(t), -\eta + \min_{[0,1]} V^+(t,\cdot) \right\} \le V^-(t,y) \le \max\left\{ V_1(t), -\eta + \max_{[0,1]} V^+(t,\cdot) \right\} \,,
\end{eqnarray*}
for $(t,y)\in\Omega_T$. These inequalities are consistent with \fref{V-est} since, by \fref{V-ass1}, \fref{V-est} is also satisfied by $(V_0,V_1)$. 
Thus, \fref{eq:ell} could also be solved by a fixed point iteration (in the subset of $L^\infty(0,1)^2$ defined by \fref{V-est}), where the problems 
for $V^+$ and $V^-$ are solved alternatingly, which proves \fref{V-est}. 
The statements about the derivatives are proven analogously to the proof of Lemma \ref{lem:Vfreeb}.
\end{proof}

The following result is concerned with the subproblem for the densities and corresponds to Lemma \ref{lem:rhofreeb} of the preceding section.
Here it also provides the basis of a global existence result under a smallness assumption on the data.

\begin{lem}\label{lem:rhofixedb}
Let the initial and boundary data for the densities, the bundle length, and the velocities satisfy \fref{rho-ass}, \fref{X-ass}, and, respectively, 
\fref{V-est}. Then \\
a) for $T>0$ small enough, the problem \fref{eq:rhofreeb} has a unique solution $\rho^\pm \in R_{\rho,T}$ 
with $\gamma>0$ depending on the data, \\
b) for arbitrary $T$ and for $X_0$ small enough, the problem \fref{eq:rhofreeb} has a unique solution $\rho^\pm$ satisfying
\begin{equation} \label{rho-est}
   \rho^\pm(t,y,l) \ge \alpha_T>0 \quad\mbox{for } l\le \underline L/2 \,,\qquad \rho^\pm \le \beta_T \,,\,\,|\p_y\rho^\pm| \le \gamma_T 
   \quad\mbox{in } \Omega_T\times (0,\infty)\,. 
\end{equation}
\end{lem}

\begin{proof}
If $V^+-\dot Xy>0$ for all $y\in[0,1]$ holds, it is guaranteed that the $y$-characteristics for the plus family are ingoing at the left boundary and outgoing at the right. Indeed, by \fref{V-est},
\[
  V^+-\dot Xy\ge \delta +\dot X_+-(\dot X_+-(-\dot X)_+)y = \delta +\dot X_+(1-y)+(-\dot X)_+y \ge \delta \,.
\]
Analogously, $V^- - \dot X y \le -\delta$ is proven. \\
a) For small $T$, the rest of the proof is as in Lemma \ref{lem:rhofreeb}. \\
b) Since, by the above estimates and by \fref{X-ass}, $\delta/(2X_0)$ is the minimum $y$-speed of the characteristics, the maximal life time of a characteristic before it leaves the bundle (i.e. the $y$-interval $[0,1]$) is $2X_0/\delta$. By \fref{rho-ass}, the initial and boundary data for 
the densities are bounded away from zero for $l\in [0,\underline L]$. The right end of this interval moves to $\underline L - s_l\tau$ on a characteristic
with life time $\tau$, and it thus stays above $\underline L/2$ for
\[
  X_0 \le \frac{\delta \underline L}{4s_l} \,,
\]
completing the proof of the first estimate in \fref{rho-est}, since the sign of the densities is preserved along characteristics. The other estimates
are proved as in Lemma \ref{lem:rhofreeb}.
\end{proof}

For the local existence result we again need Lipschitz continuity. The following result can be proved analogously to Lemma \ref{lem:Lipfreeb} and the proof 
is therefore omitted.

\begin{lem}\label{lem:Lipfixedb}
a) Let $(D^\pm_1,C_1)$ and $(D^\pm_2,C_2)$ satisfy the assumptions of Lemma \ref{lem:Vfixedb}. Then there exists 
a constant $c>0$ depending on the data, such that the corresponding solutions $V^\pm_1$ and $V^\pm_2$ of \fref{eq:ell} satisfy
\begin{eqnarray*}
  && \|\Delta V^+\|_{\rm vel}+\|\Delta V^-\|_{\rm vel}  \le c\left( \|\Delta C\|_{\rm coeff} + 
     \|\Delta D^+\|_{\rm coeff} + \|\Delta D^-\|_{\rm coeff}\right) \,.
\end{eqnarray*}
b) Let the assumption of Lemma \ref{lem:rhofixedb} hold for $V^\pm_1$ and $V^\pm_2$ and let $T$ be small enough in the sense of Lemma
\ref{lem:rhofixedb} a). Then there exist constants $c_1,c_2>0$, such that 
the corresponding solutions $\rho^\pm_1$ and $\rho^\pm_2$ of \fref{eq:rhofreeb} satisfy
$$
    \|\Delta\rho^+\|_{\rm density} + \|\Delta\rho^-\|_{\rm density} \le c_1\left(e^{c_2T}-1\right) 
  \left( \|\Delta V^+\|_{\rm vel} + \|\Delta V^-\|_{\rm vel}  \right) \,.
$$
\end{lem}

The existence results will again be written for the problem in the original scaling.

\begin{thm}\label{thm:loc-ex}
 Let $C_0,D_0,\eta,s_l,X_0>0$, \fref{rho-ass}, \fref{X-ass}, and \fref{V-ass1} hold. Then for $T>0$ small enough the problem 
\fref{eq:prob}--\fref{eq:coeff} has a unique solution 
$$
  (\rho^+,\rho^-,V^+,V^-) \in L^\infty\left(0,T;\, W^{1,\infty}_x(L^\infty_l)^2\times (W^{2,\infty}_x)^2 \right) \,,
$$
such that $\rho^+,\rho^-> 0$ are uniformly bounded away from zero in $\{(t,x,l):\, 0<t<T,\,0<x<X(t),\,0<l<\underline L/2\}$,
$V^+ > \max\{0,\dot X\}$, and $V^- < \min\{0,\dot X\}$. 
\end{thm}

\begin{proof}
Lemmas \ref{lem:CDfreeb}, \ref{lem:Vfixedb}, \ref{lem:rhofixedb} show that for small enough $T$ the fixed point operator maps $R_\rho^2$ into itself
and that the solution components have the stated properties. 
By Lemmas \ref{lem:Lipfreeb} a) and \ref{lem:Lipfixedb} the fixed point map is a contraction with respect to the 
$L^\infty_t(0,T;\, L_{y,l}^2((0,1)\times(0,\infty)))$-topology. 
\end{proof}

The only obstacle for the global existence of solutions is the possibility that all filaments, pushed in at one end
of the bundle, get completely depolymerized before reaching the other end. Lemma \ref{lem:rhofixedb} b) shows that this can be avoided
in short enough bundles.

\begin{thm}
Let the assumptions of Theorem \ref{thm:loc-ex} hold. Then, for $X_0$ small enough the problem 
\fref{eq:prob}--\fref{eq:coeff} has a unique global solution.
\end{thm}

\begin{proof}
By Lemma \fref{lem:rhofixedb} b) the qualitative properties of the initial data $\rho^\pm_I$, requested in \fref{rho-ass}, are propagated
for arbitrary times, permitting continuation of the local solution.
\end{proof}

\section{Asymptotics for small prescribed force}

In this section we show that it seems possible to improve the results of Section \ref{fbvp}.
For small force and time independent data, global solvability and convergence
to a steady state can be expected.

We assume time independent boundary data and prescribed force, 
i.e. $u_0^+(t)=u_l$, $u_1^-(t) = u_r$, $\rho_0^+(t,l) = \rho_l(l)$, $\rho_1^-(t,l) = \rho_r(l)$, 
and $F =\,$const. Furthermore, the force is assumed to be pulling, $F>0$, and the force free bundle to be contractive, $\eta> u_l+u_r$. 

With a small parameter $\eps>0$, we introduce the small force rescaling
$$
  F\to \eps F\,,\qquad X\to \eps X\,,\qquad x\to \eps x\,,\qquad t\to\eps t \,,
$$
expressing the expectation of a small bundle length of the same order of magnitude
as the force and of an accordingly faster relevant time scale. The rescaled problem
for the velocities with prescribed force reads
\begin{eqnarray}\label{eq:V-eps}
\begin{array}{ l}     
    \p_x\left(D^\pm\p_xV^\pm\right) \pm \eps^2 C \left(\eta - V^+ + V^-\right) = 0 \,,\\
    V^+(t,0)= u_l \,,\qquad \p_xV^+(t,X(t))=0 \,,\\
    \p_xV^-(t,0)=0 \,,\qquad  (D^-\p_xV^-)(t,X(t)) = \eps^2 F \,.
\end{array}
\end{eqnarray}
The dynamics of the bundle length is then determined by
\begin{equation}\label{X-eps}
\dot X(t) = V^-(t,X(t)) + u_r \,.
\end{equation}
We shall pass to the limit $\eps\to 0$ formally and denote formal limits by the subscript $0$.
Obviously, the formal limit of \eqref{eq:V-eps} has the solution
$V^+_0 = u_l$, $V^-_0 = V^-_0(t)$ with the latter to be determined by integration 
of the differential equation for $V^-$, substitution of the last boundary condition,
 division by $\eps^2$, and passing to the limit:
$$
  V^-_0(t) = \frac{F}{\int_0^{X_0(t)} C_0(t,x)dx} + u_l - \eta \,.
$$
The formal limit
\begin{eqnarray}\label{eq:rho-lim}
\begin{array}{ l}     
    \p_t\rho^\pm_0 + V^\pm_0\p_x\rho^\pm_0 = 0 \,,\\
    \rho^+_0(t,0,l)= \rho_l(l) \,,\qquad \rho^-_0(t,X(t),l)=\rho_r(l) \,,
\end{array}
\end{eqnarray}
of the equations for the densities has -- after finite time -- the solution $\rho^+_0= \rho_l$, $\rho^-_0=\rho_r$,
producing a constant coefficient $C_0$ and thus $V^-_0 = \frac{F}{C_0X_0} + u_l - \eta$. The limiting dynamics of the bundle length is therefore given by the ODE
$$
  \dot X_0 = u_l+u_r-\eta + \frac{F}{C_0X_0}
$$
with the stable positive steady state
$$
  X_\infty = \frac{F}{C_0(\eta - u_l - u_r)} \,.
$$
We conjecture that the asymptotics of this section can be made rigorous in a rather straightforward
way by tracking the dependence on $\eps$ in the estimates of Section \ref{fbvp}.
The convergence to equilibrium will not hold for large forces, which might rip the bundle apart.
\par\bigskip\noindent
\textbf{Email adresses}: \\
stefanie.hirsch$@$univie.ac.at,\\
dietmar$@$cims.nyu.edu,\\
christian.schmeiser$@$univie.ac.at.


\begin{thebibliography}{99}

\bibitem{Carvalho} A. Carvalho, A. Desai, K. Oegema, Structural Memory in the Contractile Ring Makes the Duration of Cytokinesis Independent of Cell Size, ``Cell``, 137(5) (2009), 926--937.
\bibitem{Courson} D.S. Courson, R.S. Rock, Actin cross-link assembly and disassembly mechanics for $\alpha$-actinin and fascin, ''J. Biol. Chemistry``, 285 (2010), 26350--26357.
\bibitem{Huxley} A.F. Huxley, Muscular contraction, ''J. Physiology``, 243 (1974), 1--43.
\bibitem{Jayo} A. Jayo, M. Parsons, Fascin: A key regulator of cytoskeletal dynamics, ''Int. J.  Biochem. Cell Biol.`` , 42(10) (2010), 1614--1617.
\bibitem{MilOel} V. Milisic, D. \"Olz, On the asymptotic regime of a model for friction mediated by transient elastic linkages, ''J. Math. Pures Appl.``, 96(5) (2011), 484--501.
\bibitem{Oelz} D. \"Olz, A viscous two-phase model for contractile actomyosin bundles, ''Math. Biol.``, 68 (2013), 1653--1676.
\bibitem{OelSch} D. \"Olz, C. Schmeiser, Derivation of a model for symmetric lamellipodia with instantaneous cross-link turnover, ''Arch. Rational Mech. Anal.``, 198(3) (2010), 963--980.
\bibitem{Small} J.V. Small, K. Rottner, I. Kaverina, K.I. Anderson, Assembling an actin cytoskeleton for cell attachment and movement, ''Biochimica et Biophysica Acta - Molecular Cell Research``, 1404 (1998), 271--281.
\bibitem{Svitkina} T.M. Svitkina, A.B. Verkhovsky, K.M. McQuade, G.G. Borisy, Analysis of the actin-myosin II system in fish epidermal keratocytes: Mechanism of cell body translocation, ''J. Cell Biol.``, 139 (2) (1997), 397--415. 
\bibitem{Vicente} M. Vicente-Manzanares, X. Ma, R.S. Adelstein, A.R. Horwitz, Non-muscle myosin II takes centre stage in cell adhesion and migration, ''Nature Rev. Mol. Cell Biol.``, 10 (2009), 778--790.
\end{thebibliography}
\end{document}